\documentclass[12pt]{article}
\usepackage{amsfonts,amsmath,color,amsthm,amssymb}
\usepackage{tikz}
\usepackage{pstricks}
\usepackage{graphicx}
\usepackage{epstopdf}
\usepackage{permute}
\usepackage{latexsym}
\usepackage{mathrsfs}
\usepackage{hyperref}

\title{Universal and Near-Universal Cycles of Set Partitions}

\author{Zach Higgins\\
Department of Mathematics\\ University of Florida\and 
Elizabeth Kelley\\
Department of Mathematics\\ Harvey Mudd College\and
Bertilla Sieben\\ Department of Mathematics\\Princeton University\and Anant Godbole\\
Department of Mathematics and Statistics\\ East Tennessee State University
}
	
\begin{document}
	\maketitle
\def\l{\mathbb L}
\def\la{\lambda}
\def\lr{\left(}
\def\rr{\right)}
\def\lc{\left\{}
\def\rc{\right\}}
\def\vp{\varphi}
\def\ca{\mathcal A}
\def\cp{\mathcal P}
\def\b{\beta}
\def\a{\alpha}
\def\e{\mathbb E}
\def\p{\mathbb P}
\def\v{\mathbb V}
\def\lg{{\rm lg}}
\def\ra{\rightarrow}
\def\nm{{n\choose m}}
\def\mr{{m\choose r}}
\def\nmr{{{n-m}\choose{m-r}}}
\newcommand{\Z}{\mathbb Z}
\newcommand{\marginal}[1]{\marginpar{\raggedright\scriptsize #1}}
\newcommand{\BS}{\marginal{BS}}
\newcommand{\AG}{\marginal{AG}}
\newcommand{\EK}{\marginal{EK}}
\newcommand{\ZH}{\marginal{ZH}}
\newtheorem{thm}{Theorem}
\newtheorem{lm}[thm]{Lemma}
\newtheorem{cor}[thm]{Corollary}
\newtheorem{rem}[thm]{Remark}
\newtheorem{exam}[thm]{Example}
\newtheorem{prop}[thm]{Proposition}
\newtheorem{defn}[thm]{Definition}
\newtheorem{cm}[thm]{Claim}

	\begin{abstract}
			We study universal cycles of the set $\cp(n,k)$ of $k$-partitions of the set $[n]:=\{1,2,\ldots,n\}$ and prove that the transition digraph associated with $\cp(n,k)$ is Eulerian.  But this does not imply that universal cycles (or ucycles) exist, since vertices represent equivalence classes of partitions!  We use this result to prove, however, that ucycles of $\cp(n,k)$ exist for all $n \geq 3$ when $k=2$.  We reprove that they exist for odd $n$ when $k = n-1$ and that they do not exist for even $n$ when $k = n-1$.  An infinite family of $(n,k)$ for which ucycles do not exist is shown to be those pairs for which $S(n-2, k-2)$ is odd ($3 \leq k < n-1$).  We also show that there exist universal cycles of partitions of $[n]$ into $k$ subsets of distinct sizes when $k$ is sufficiently smaller than $n$, and therefore that there exist universal packings of the partitions in $\cp(n,k)$.  An analogous result for coverings completes the investigation.   
		\end{abstract}
		
	\section{Introduction}  Universal cycles are often loosely defined. In a recent seminar talk at Viriginia Commonwealth University, Glenn Hulbert offered the following description: ``Broadly, universal cycles (ucycles) are special listings of combinatorial objects
in which codes for the objects are written in an overlapping, cyclic manner."  By ``special," Hurlbert means ``without repetitions".  As an example, the cyclic string 112233 encodes each of the six multisets of size 2 from the set $\{1,2,3\}$. Another well-quoted example, from \cite{h}, is the string
\[ 1356725\ 6823472\ 3578147\ 8245614\ 5712361\ 2467836\ 7134582\ 4681258, \]
 where each block is obtained from the previous one by addition of 5 modulo 8. This string is an encoding of the fifty six 3-subsets of the set $[8]:=\{1,2,3,4,5,6,7,8\}$.  A seminal paper in the area is that of Chung, Diaconis and Graham \cite{cdg} who studied ucycles of 
\begin{itemize}\item subsets of size $k$ of an $n$-element set (as in the above example); \item set partitions (the focus of this paper); and \item permutations (with a necessarily augmented ground set and the use of order isomorphism representations, e.g., the string 124324 encodes each of the six permutations of $[3]=\{1,2,3\}$ in an order isomorphic fashion, which is clearly not possible to do using the ground set $[3]$).\end{itemize} 
  In \cite{cdg} it is shown that for $n\ge 4$, there {\it does exist} a ucycle of all partitions $\cp(n)$ of the set $[n]$ into an arbitrary number of parts; for example, we have the ucycle $abcbccccddcdeec$ of $\cp(4)$, where, for example, the substring $dcde$ encodes the partition $13\vert2\vert4$. Note that the alphabet used was in this case of size 5, though an alphabet of (minimum) size 5 is shown to suffice to encode $\cp(5)$ as 
$$DDDDDCHHHCCDDCCCHCHCSHHSDSSDSSHSDDCH$$$$SSCHSHDHSCHSJCDC.$$
The above example reflects tongue-in-cheek humor, since there are 52 partitions of $[5]$ and the above ucycle has 13 cards of each suit -- except that one spade has been replaced by a joker!  The authors also ask how many partitions of $\cp(n)$ using an alphabet of size $N\ge n$ exist. This question will also be of deep relevance to us, as alluded to in the later section on Open Problems.

As noted in \cite{ds}, however, not much seems to be known about ucycles of the partitions $\cp(n,k) (\vert\cp(n,k)\vert=S(n,k))$ of $[n]$ into $k$ parts.  In \cite{cs}, it was shown that for $k=n-1$, ucycles exist if and only if $n$ is odd.  At the other end of the $k$-spectrum, the authors of \cite{egm} showed that for odd $n$, one could find a ucycle of partitions of $[n]$ into two parts, and that an ``asymptotically good universal packing" could be found for $k=3$, i.e., that there was a string of length $T(n,3)<S(n,3)$ with \[\frac{T(n,3)}{S(n,3)}\to 1, \quad n\to\infty,\]
where each of the $T(n,3)$ consecutive strings of length $n$ represented a different partition of $[n]$ into 3 parts of distinct sizes.  The authors of \cite{egm} also proved that ucycles of partitions of $[n]$ into 2 parts exist if $n=4$ or $6$, as evidenced by the respective explicit examples
$$aabbaba$$ and $$abaaaabbaababbbabbaaabbbbbababa.$$ It is this work that we build on.  In Section 2, we generalize the above result on asymptotically good universal packings (upackings) to the case of all fixed $k$ as $n\to\infty$, as well as to ucoverings, which are also shown to be ``asymptotically good."    Finally, Section 3 is devoted to showing that the transition digraph associated with $\cp(n,k)$ is Eulerian. As noted in the work of \cite{cdg}, this does not necessarily imply that universal cycles exist, since the digraph vertices represent equivalence classes of partitions.  We use our result to prove, however, that ucycles of $\cp(n,k)$ exist for all $n \geq 3$ when $k=2$ and for odd $n$ when $k = n-1$, the latter recovering the result in \cite{cs}.  We also (re)prove that ucycles do not exist for $n$ even when $k = n-1$.  Finally, we show that for even $n$, ucycles do not exist when $S(n-2, k-2)$ is odd ($3 \leq k < n-1$).  The last result shows, e.g.,  that ucycles of $\cp(12,6)$, or $\cp(6,3)$ do not exist.  There are infinitely many such pairs of values of $(n,k)$.  Moreover, the technique we exhibit  in Section 3 has the potential to tease out many more results along these lines.

\section{Universal Packings and Coverings of Partitions of $[n]$ into $k$ parts} One of the main results in \cite{lbg} was that one could create a
ucycle of all surjections from $[n]$ to $[k]$ iff $n>k$. Since
there are $k!$ surjections that yield the same set partition, we
need to be more careful, and proceed by showing in Theorem 1 that
 for sufficiently large $n$, it is possible to ucycle
partitions of $[n]$ into $k$ parts of {\em distinct} sizes. We
represent such partitions as surjections from $f:[n]\rightarrow[k]$;
$n>k$, with $\vert f^{-1}(\{1\})\vert<\vert f^{-1}(\{2\})\vert<\ldots<\vert f^{-1}(\{k\})\vert$.
The fact that \textit{asymptotically good upackings} exist is proved
in Theorem 2, where we provide an alternative proof of the
intuitively obvious fact that for any fixed $k$, the number of $k$-partitions
of $[n]$ with non-distinct part sizes form a vanishing fraction of
all partitions into $k$ parts as $n\to\infty$. \begin{thm} For
each fixed $k,n;n\ge\frac{(k+4)(k-1)}{2}+1$, there exists a ucycle
of all onto functions $f:[n]\rightarrow[k]$ such that the preimage cardinality 
function $\vert f^{-1}\vert:\{\{1\},\{2\},\ldots\{k\}\}\rightarrow[n]$
is strictly increasing. \end{thm} \begin{proof} Following the standard
process, we create a digraph $D$ in which the vertices are sequences of
length $n-1$, of numbers in $\{1,\ldots,k\}$, for which the addition
of at least one number in $\{1\ldots,k\}$ at the beginning or the
end creates a sequence of length $n$ which, using the special canonical
format we have adopted, represents a partition of $[n]$ into $k$
parts of distinct sizes. For example, with $n=10,k=4$, 122333444
is a legal vertex, as is 123334444. On the other hand, 112233344 is
not in the underlying graph. A vertex points towards another if its
last $n-2$ terms are the same as the first $n-2$ terms of the second
vertex. The edges of the digraph, labeled by vertex label concatenation,
are thus sequences representing partitions of $[n]$ into $k$ distinct
parts.

The problem of finding a ucycle is reduced to the problem of finding
an Eulerian circuit in this digraph. We know that Eulerian circuits
exist if the graph is both strongly connected and balanced, i.e.,
for each vertex $v$, the in- and out-degrees of $v$ are equal. 
It is easy to show that if a digraph is balanced and weakly connected,
then it is also strongly connected, so all we need to show is that
our digraph is balanced and weakly connected. This approach
is used, e.g., in \cite{bg}.

Showing that $D$ is balanced is simple - if a number from 1 through
$k$ can be added to the beginning of the sequence at a vertex, then
it can also always be added to the end of the sequence since it is
only the numbers of 1's, 2's, $\ldots$ and $k$'s that actually matter
in determining if an edge represents a partition into distinct parts.
Therefore, the number of edges pointing away from a vertex will be
the same as the number of edges pointing towards it. Note that the
in- and out-degrees of any vertex (the common value is sometimes called
the  vertex degree), though equal, are quite different at different
vertices. For example for $k=3$, the vertex 123333333 has degree
1; the vertex 122333333 has degree 2; and the vertex 122233333 has
degree 3. In general, one may write down a formula for $deg(v)$ depending
on the differences between the number of $i+1$s and the number of
$i$s in $v$.

To show that the digraph is weakly connected, we need only show that
it is possible to reach a designated target vertex from any other
starting vertex. We will let this target vertex be the one consisting
of two 2's, three 3's,$\ldots$,and $k-1$ $k-1$'s -- leaving all
of the remaining numbers as $k$'s, in that order. For example, for
$n=27$ and $k=6$, the target vertex of length 26 will be $22333444455555666666666666$.
In fact, we will show something stronger, namely that one can
traverse from any edge to the edge $122333\ldots(k-1)\ldots(k-1)k\ldots k$
from whence we may reach the target vertex in a single step. Notice
that such edges represent legal partitions in $\cp(n,k)$ only if
$n\ge n_{0}:=k(k+1)/2$.

Key to our algorithm on how to reach one edge from another is the
process of ``switching" numbers. For example, we can
write out all of the steps to go from the edge $33132323$ to the
edge $33122323$ as follows $33132323\ra31323233\ra13232333\ra32323331\ra23233312\ra32333122\ra23331223\ra33312232\ra33122323,$
or, we can skip all the steps of ``rotating around"
and just say that we ``switched" the 3 to the right
of the 1 into a 2. Note that ``rotating" is always
legal but switching in the above form might not always be, even in
several steps. We need to have the ``room to maneuver around"
while maintaining edge-integrity.

We will show that the only requirement to reach the designated target
edge from any other edge is to have the ability to ``switch"
any $j\in\{2,3,\ldots,k-1\}$ into a $k$ and vice versa (possibly
through several steps), and that this is equivalent to having $n\ge n_{1}:=\frac{(k+4)(k-1)}{2}+1=n_{0}+(k-1)$,
where the ``extra" $(k-1)$ digits give us the needed
room to maneuver around. %This is possible whenever we can increase the number of 2's by at least one without breaking the rules. 

We will let the partition size vector (PSV) of an edge be a $k$-tuple
that expresses the number of 1's, 2's, etc. in order. For example,
with $n=18;k=5$, the only possible PSV's are $(1,2,3,4,8)$; $(1,2,3,5,7)$;
and $(1,2,4,5,6)$. In this case, there is no way, e.g., to change
a 2 with any other number. If, for example, a 2 is switched
with a 5, this forces the numbers of 1's and 2's to be equal, and
other issues might arise if a 2 is switched with a 3 or a 4. However,
if we have $n=19$ and $k=5$, the PSV's are $(1,2,3,4,9)$, $(1,2,3,5,8)$,
$(1,2,3,6,7)$, $(1,2,4,5,7)$, and $(1,3,4,5,6)$, which will be
seen to imply that no matter what our starting position, we have enough
``spaces" so that we can switch a 2, 3, or a 4 into
a 5 (possibly in multiple steps) and back again eventually. Notice
that for $k=5$, $n_{1}=19$. For example, say we begin with the edge
$(1,3,4,5,6)$ and we want to switch a 3 in a particular spot into
a 2, while maintaining the PSV composition. We will do this by first
switching the 3 to a 5, then switching the 5 to a 2, in several steps:

a. Switch any 2 into a 5 to create space between the number
of 2's and the number of 3's, we will then have PSV=($1,2,4,5,7)$;

b. Switch the 3 that we want to eventually switch into a 2 
into a 5, we then have PSV=$(1,2,3,5,8)$;

c. Switch a \textit{different } 5 back into a 3 to create space
between the number of 2's and the number of 3's, PSV=$(1,2,4,5,7)$;
and finally 

d. Switch the 5 we want into a 2, yielding the switch we originally
wanted. PSV=$(1,3,4,5,6)$.

In general, our algorithm to reach the target sequence will then be
to first, if we have more than one 1, change all extra 1's into $k$'s.
We will then underline the single remaining 1 as something we won't
touch again. Next, we will switch the number to the right of the 1
into a $k$, possibly in multiple steps, and then the $k$ back into
a 2, again possibly through multiple steps. We will now underline
the 1 and the 2 together, as something we won't touch again. Next
we will consider the next number to the right of this 2 and switch
it to a $k$, and then switch back from the $k$ into a 2 again, then
underline the sequence $122$, and switch all remaining 2's in the
sequence into $k$'s. Next, we consider the number to right of the
second 2, switch it into a $k$, and then switch from the $k$ back
into a 3, then block off $1223$, etc., until we reach the target sequence.
The following example illustrates the general technique.

Let $n=19,k=5$. Suppose we begin with the sequence

\noindent $1432543552543455435$ with PSV=$(1,2,4,5,7)$. There are
no extra 1's to change into 5's. Then we can first change the 4 to
the right of the 1 into a 5, but in order to do this, we must create
space between the numbers of 3's and 4's by changing one 3 into a
5; we thus arrive at $1452543552543455435$ PSV=$(1,2,3,5,8)$. The
4 to the right of the 1 can now be changed into a 5 to get $1552543552543455435$
PSV=$(1,2,3,4,9)$. We now need to change the 5 to the right of the
1 back into a 2 by creating space between the number of 4's and 3's
and then between the number of 3's and 2's. This leads to 
\[
1542543552543455435;(1,2,3,5,8);
\]
 
\[
1542343552543455435;(1,2,4,5,7);
\]
 and finally 
\[
\underline{12}42343552543455435;(1,3,4,5,6),
\]
 where the vectors above all represent PSV's. We then change the next
4 into a 5 and then back into a 2 as follows: 
\[
\underline{12}45343552543455435;(1,2,4,5,7),
\]
 
\[
\underline{12}45543552543455435;(1,2,3,5,8),
\]
 
\[
\underline{12}55543552543455435;(1,2,3,4,9),
\]
 
\[
\underline{12}54543552543455435;(1,2,3,5,8),
\]
 
\[
\underline{12}54343552543455435;(1,2,4,5,7),
\]
 and 
\[
\underline{122}4343552543455435;(1,3,4,5,6).
\]
 The rest of the algorithm proceeds as follows: 
\[
\underline{122}4343555543455435;(1,2,4,5,7),
\]
 
\[
\underline{122}4543555543455435;(1,2,3,5,8),
\]
 
\[
\underline{122}5543555543455435;(1,2,3,4,9),
\]
 
\[
\underline{122}5443555543455435;(1,2,3,5,8),
\]
 
\[
\underline{1223}443555543455435;(1,2,4,5,7),
\]
 
\[
\underline{1223}445555543455435;(1,2,3,5,8),
\]
 
\[
\underline{1223}545555543455435;(1,2,3,4,9),
\]
 
\[
\underline{1223}544555543455435;(1,2,3,5,8),
\]
 
\[
\underline{12233}44555543455435;(1,2,4,5,7),
\]
 
\[
\underline{12233}44555545455435;(1,2,3,5,8),
\]
 
\[
\underline{12233}54555545455435;(1,2,3,4,9),
\]
 
\[
\underline{12233}54455545455435;(1,2,3,5,8),
\]
 
\[
\underline{12233344}55545455435;(1,2,4,5,7),
\]
 
\[
\underline{12233344}55545455455;(1,2,3,5,8),
\]
 
\[
\underline{122333444}5545455455;(1,2,3,6,7),
\]
 
\[
\underline{122333444}5555455455;(1,2,3,5,8),
\]
 
\[
\underline{1223334444555}455455;(1,2,3,6,7),
\]
 
\[
\underline{1223334444555555555};(1,2,3,4,9).
\]
 This process works in general since, given an edge of weight $n_{1}:=\frac{(k+4)(k-1)}{2}+1$,
the sums of the gaps between the components of the PSV may be as low
as $k$, corresponding to the PSV $(1,3,4,\ldots,k+1)$, or as high
as $2k-1$, corresponding to the PSV $(1,2,\ldots,k-1,2k-1)$. Switching
numbers, possibly in multiple steps, is always possible whenever there
is a gap of two somewhere in the PSV sequence, which is guaranteed
by the choice of $n_{1}$. Question: Does a better algorithm allow
for a smaller threshold $n$? \end{proof} 
The next theorem shows that there exists a ucovering of all partitions of $[n]$ into $k$ parts if $n$ is large enough; for simplicity we let the threshold $n$ be the same as in Theorem 1.  This is because any partition may be represented by a surjection satisfying the conditions of Theorem 2, though there may be multiple such representations when two or more of the part sizes are equal.
\begin{thm} For each fixed $k,n; n\ge\frac{(k+4)(k-1)}{2}+1$, there exists a ucycle of all onto functions $f:[n]\rightarrow[k]$ such that the preimage cardinality function $\vert f^{-1}\vert:\{\{1\},\{2\},\ldots\{k\}\}\rightarrow[n]$ is non-decreasing.
\end{thm}
\begin{proof}
Exactly the same as that of Theorem 1, except that the algorithm may terminate faster.
\end{proof}
\begin{cor} For any fixed $k\ge3$, the upacking and ucovering given in Theorems 1 and 2 respectively are both asymptotically of size $S(n,k)$, the number of partitions of $[n]$ into $k$ parts.
\end{cor}
\begin{proof} Our proof will reveal that 
$$\frac{T(n,k)}{S(n,k)}=1-O\lr\frac{1}{\sqrt{n}}
\rr; \frac{U(n,k)}{S(n,k)}=1+O\lr\frac{1}{\sqrt{n}}\rr,\enspace n\to\infty,$$ where $T(n,k)$ and $U(n,k)$ are the lengths of the ucycles in Theorems 1 and 2 respectively.  As pointed out by Professor L\'aszl\'o Sz\'ekely, however, both these results are special cases of asymptotic results found in \cite{cms}, where a threshold of $k=n^{1/5}$ is seen to hold for the property ``partitions of size $k$ with distinct parts form a ``high" fraction of all partitions of $[n]$ into $k$ parts."  Thus such values of $k$ can serve to improve the conclusion of Corollary 3.  Our proof is somewhat different, however.

Note that $$\frac{T(n,k)}{S(n,k)} = 1 - \frac{Sa(n,k)}{S(n,k)},$$ where $Sa(n,k)$ denotes the number of partitions in which two or more parts are equal. 
Reframing the question in terms of distributing $n$ distinguishable balls into $k$ distinct boxes, we would like to calculate
$\p \left(\cup_{i,j} I_{i,j} \right)$
where $I_{i,j} = \{B_i = B_j \},$ is the event that boxes $i$ and $j$ contain the same number of balls $B_i$ and $B_j$ respectively. By symmetry, we observe that
\begin{align}
\p\left(\bigcup_{i,j} I_{i,j} \right) &\le {k \choose 2} P(B_1 = B_2) \nonumber\\
&= {k \choose 2} \sum_{j= 0}^{\lfloor n/2 \rfloor} \frac{{n \choose j}{n-j \choose j}}{k^n} (k-2)^{n-2j}\nonumber\\
&={k \choose 2}\left(\frac{k-2}{k} \right)^n \sum_{j=0}^{\lfloor n/2 \rfloor} {n \choose j}{n-j \choose j} \frac{1}{(k-2)^{2j}}.
\end{align}
\begin{lm} The function $f(j) = {n \choose j}{n-j \choose j}\frac{1}{(k-2)^{2j}}$ obtains its maximum value when $j = \frac{n}{k}$.\end{lm}
\begin{proof}

Parametrize by setting $j = An$, we see  (by Stirling's formula) that
\begin{align*}
&{n \choose An}{(1-A)n \choose An}\frac{1}{(k-2)^{2An}} \\&= \frac{n!}{(An)!(An)!((1-2A)n)!}\frac{1}{(k-2)^{2An}} \\
&=\frac{K+o(1)}{n}\left(\frac{n}{e} \right)^{n} \left(\frac{e}{An} \right)^{2An} \left(\frac{e}{(1-2A)n} \right)^{(1-2A)n} \frac{1}{(k-2)^{2An}} \\
&= \frac{K}{nA^{2An}(1-2A)^{(1-2A)n}(k-2)^{2An}}(1+o(1)).
\end{align*}
Maximizing this function is equivalent to minimizing the natural log of its denominator. Accordingly, define
$$
\beta(A)= 2An\ln A + (1-2A)n \ln (1-2A) + 2An \ln (k-2).$$
Setting $\beta'(A) = 0$, we find easily that $A=\frac{1}{k}$.
The next step is to show that  $\frac{f(j+1)}{f(j)} < 1$ for $j > \frac{n}{k}$ and $\frac{f(j+1)}{f(j)} > 1$ for $j < \frac{n}{k}$.  It is routine to calculate that
$$
\frac{f(j+1)}{f(j)} = \frac{(n-2j)(n-2j-1)}{(k-2)^2(j+1)^2},
$$
and thus that for any $\epsilon>0$,
$$
\frac{f\left(\frac{n}{k} + \epsilon + 1 \right)}{f\left(\frac{n}{k} + \epsilon \right)} 
= \frac{\left(\frac{n}{k} - \frac{2\epsilon}{k-2} \right)\left(\frac{n}{k} - \frac{2\epsilon + 1}{k-2} \right)}{\left(\frac{n}{k} + \epsilon + 1 \right)^2} 
< 1,
$$
and 
$$
\frac{f\left(\frac{n}{k} - \epsilon + 1 \right)}{f\left(\frac{n}{k} - \epsilon \right)} = \frac{\left(\frac{n}{k} + \frac{2\epsilon}{k-2} \right)\left(\frac{n}{k} + \frac{2\epsilon - 1}{k-2} \right)}{\left(\frac{n}{k} - \epsilon + 1 \right)^2} 
> 1.
$$
The lemma follows.
\end{proof}
\vspace{3mm}
We now return to (1) and see that for a $\vp(n)$ to be determined,
\begin{eqnarray}
&{}&\p\left(\bigcup_{i,j} I_{i,j} \right)\nonumber\\ &\leq& {k \choose 2}\left(\frac{k-2}{k} \right)^{n} \left[\sum_{j = \frac{n}{k} - \vp(n)}^{\frac{n}{k} + \vp(n)}f\left(\frac{n}{k}\right) + \sum_{j = 0}^{\frac{n}{k}-\vp(n)} f(j) + \sum_{j = \frac{n}{k} + \vp(n)}^{\lfloor n/2 \rfloor} f(j) \right] \nonumber\\
&\leq& {k \choose 2}
\left(\frac{k-2}{k} \right)^n \cdot\nonumber\\&&\left[2\vp(n)f\left(\frac{n}{k} \right) + \left(\frac{n}{2} - 2\sqrt{n} \right)\max\left(f\left(\frac{n}{k} + \vp(n) \right),f\left(\frac{n}{k} - \vp(n) \right) \right)\right].
%&= {k \choose 2}\left(\frac{k-2}{k} \right)^n\cdot\\& \left[2\sqrt{n}\frac{{n \choose n/k}{(k-1)n/k \choose n/k}}{(k-2)^{2n/k}} + \left(\frac{n}{2}-2\sqrt{n} \right)\frac{{n \choose n/k+\sqrt{n}}{(k-1)n/k-\sqrt{n} \choose n/k+\sqrt{n}}}{(k-2)^{2n/k+2\sqrt{n}}} \right]
\end{eqnarray}
We will next use Stirling's approximation $N!\sim{\sqrt{2\pi N}}(N/e)^n$ at various points.  Note that whenever $N=\frac{n}{k}+o(n)$ we have that ${\sqrt{2\pi N}}=\Theta{\sqrt n}$.  Accordingly, we first see that for some constant $A$
\begin{eqnarray*}
f\lr\frac{n}{k}\rr&\le&\frac{A}{n}\lr\frac{n}{e}\rr^n\lr\frac{ke}{n}\rr^{2n/k}\lr\frac{ke}{(k-2)n}\rr^{(k-2)n/k}\lr\frac{1}{k-2}\rr^{2n/k}\\
&=&\frac{A}{n}\lr\frac{k}{k-2}\rr^{n},\end{eqnarray*} 
so that first component of (2) is no more than $B\frac{\vp(n)}{n}$.
Next notice that
\begin{eqnarray*}&&f\lr\frac{n}{k}+\vp(n)\rr\\&\le&\frac{C}{n}\lr\frac{n}{e}\rr^n\lr\frac{1}{k-2}\rr^{\frac{2n}{k}+2\vp(n)}\lr\frac{e}{n((1/k)+(\vp(n)/n))}\rr^{\frac{2n}{k}+2\vp(n)}\cdot\\&&\lr\frac{e}{n(1-(2/k)-(2\vp(n)/n)}\rr^{n-(2n/k)-2\vp(n)}\\
&=&\frac{C}{n}\lr\frac{k-2}{k}+\frac{(k-2)\vp(n)}{n}\rr^{-(2n/k)-2\vp(n)}\lr\frac{k-2}{k}-\frac{2\vp(n)}{n}\rr^{-n+(2n/k)+2\vp(n)}\\
&=&\frac{C}{n}\lr\frac{k}{k-2}\rr^n\lr1+\frac{k\vp(n)}{n}\rr^{-(2n/k)-2\vp(n)}\lr1-\frac{2k}{k-2}\frac{\vp(n)}{n}\rr^{-n+(2n/k)+2\vp(n)}.
\end{eqnarray*}
Thus the second part of (2), when using $n/k+\vp(n)$, is bounded above by
\begin{eqnarray*}
&&D \lr1+\frac{k\vp(n)}{n}\rr^{-(2n/k)-2\vp(n)}\lr1-\frac{2k}{k-2}\frac{\vp(n)}{n}\rr^{-n+(2n/k)+2\vp(n)}\\
&\le&E\exp\lc-2\vp(n)-\frac{2k}{n}\vp^2(n)+\frac{2k}{k-2}\vp(n)-\frac{4}{k-2}\vp(n)-\frac{4k}{k-2}\frac{\vp^2(n)}{n}\rc\\
&=&E\exp\lc-\frac{2k}{n}\vp^2(n)-\frac{4k}{k-2}\frac{\vp^2(n)}{n}\rc,
\end{eqnarray*}
which tends to zero provided that $\vp(n)={\sqrt{n\psi(n)}}$ for any $\psi(n)\to\infty$ (noting that $k\ge 3$ is fixed).

Finally it is easy to verify that the second part of (2) tends to zero if we consider $f(n/k-\vp(n))$ as well.  This completes the proof.
\end{proof}

\section{Universal Cycles of Partitions $\cp(n,k)$ of $[n]$ into $k$ parts}    %We use \emph{character} to refer to any letter in a string, and \emph{symbol} to refer to the elements of the alphabet of the string.  Finally, the position of the characters in the string is called their \emph{order}, and we say the \emph{relative order} of the symbols in the string is their order if we only consider the first appearance of each symbol.  Note that if all symbols in the alphabet are used then the relative order of the symbols is precisely a permutation of the symbols.  For example, the string 2123 contains the characters 2, 1, 2, and 3 with the character 2 in the first and third positions, 1 in the second position, and 3 in the fourth position (i.e. the order 2123).  It contains the symbols 1, 2, and 3 in the relative order 213.
%\marginal{**took out opening paragraph}

As in the previous section, we encode a $k$-partition of $[n]$ as a string of length $n$ containing $k$ symbols where $i$ and $j$ are in the same subset of the partition if and only if the $i$th character in the string is the same symbol as the $j$th character.  %For example, 2123 represents the partition $13|2|4$ of $[4]$.
Since the cases for $k=1$ and $k = n$ are trivial, we always assume that $2 \leq k < n$. 
% \marginal{**took out useless example, rewrote following sentence, added sentence}
For convenience, we use $\{1,2,\ldots, k\}$ as our alphabet.  We refer to an encoding of a partition as a \emph{representation} of that partition.  Note that each $k$-partition of $[n]$ has $k!$ different representations. 
%\marginal{**shortened sentence}%, and each representation corresponds to a unique relative order of the symbols (i.e. a unique permutation).

%\marginal{Deleted lengthy construction of $G_{n,k}$ and replaced with this somewhat shorter paragraph, **rewrote again more clearly}
Following methods outlined in ~\cite{cdg}, we construct a transition digraph $G_{n,k}$ for $\cp(n,k)$ as follows.  Let the set of vertices of $G_{n,k}$ be the set of all $k$ and $(k-1)$-partitions of $[n-1]$.  There is an edge between two vertices $v$ and $w$ of $G_{n,k}$ if and only if $w$ can immediately follow $v$ in a ustring of $k$-partitions of $[n]$.  That is, there is an edge from $v$ to $w$ if and only if the last $n-2$ symbols of a representation for $v$ match the first $n-2$ symbols of a representation for $w$ \emph{and} the string formed by overlaying these two representations at their shared $n-2$ length substring is a representation of a $k$-partition of $[n]$.  Observe that each vertex that is a $k$-partition of $[n-1]$ will have indegree = outdegree = $k$, and each vertex that is a $(k-1)$-partition of $[n-1]$ will have indegree = outdegree = $1$.  As an example, $G_{5,3}$ is shown in Figure ~\ref{fig:G_5,3}, with all vertices labeled with the representation having  symbols appearing in the order 123.
%For a vertex $v$ which is a $k$-partition of $[n-1]$, draw an edge from $v$ to another vertex $w$ if $w$ can immediately follow $v$ in a ustring, that is, if the last $n-2$ characters of a representation for $v$ match the first $n-2$ characters of a represenation for $w$.  We associate with the edge $vw$ the $k$-partition of $[n]$ which results when we overlap the representations for $v$ and $w$ at their shared $n-2$ length substring.  For a vertex $v$ which is a $k-1$-partition of $[n-1]$, we repeat the above construction with the additional requirement that the edge $vw$ represent a $k$-partition of $[n]$ (note that this was automatic in the previous case).  In this case, there will be exactly one edge from $v$ which connects it to the vertex $w$ represented by the string formed by appending the missing character from a representation of $v$ to the representation's last $n-2$ characters.  As an example, the graph $G_{5,3}$ with all vertices labeled with the representation with relative order 123 is shown in Figure ~\ref{fig:G_5,3}.

%\begin{figure}[h]
	%\centering
	%\includegraphics[width=1.0\textwidth, trim = 0cm 3cm 1.5cm 2cm, clip = true]{RepresentationIllustration}
	%\caption{The correspondence between vertex representations}
	%\label{fig:representations}
%\end{figure}
%\marginal{**should I remake this graph with the vertices labelled 12|3|4 style?}
\begin{figure}[h]
	\centering
	\includegraphics[width=1.0\textwidth, trim = 1.5cm 0cm 1.5cm 0cm, clip = true]{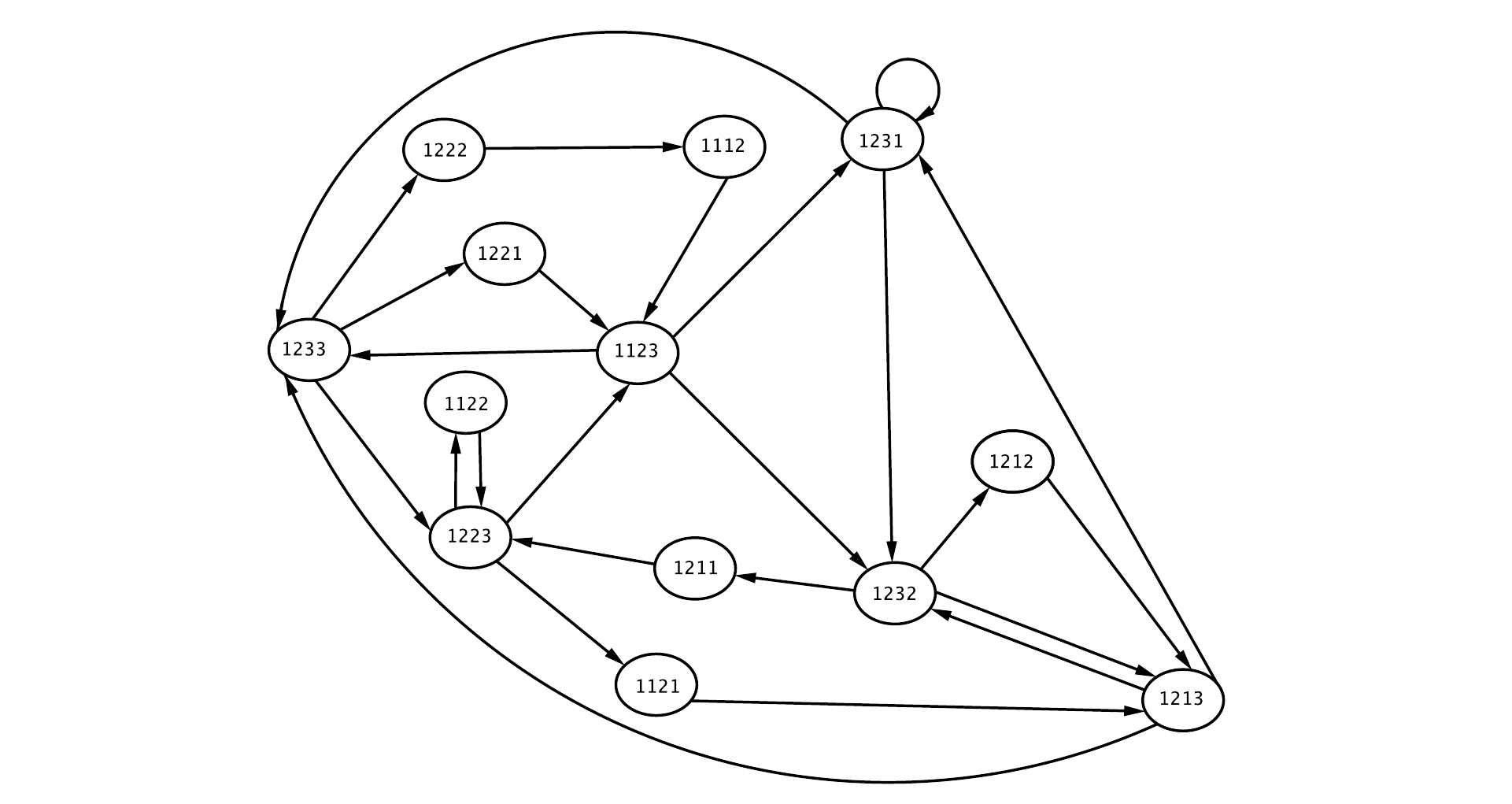}
	\caption{$G_{5,3}$}
	\label{fig:G_5,3}
\end{figure}

%\marginal{reworded this paragraph to shorten} 
Now, the edges of $G_{n,k}$ are precisely the $k$-partitions of $[n]$.  Furthermore, a partition $p_1$ can follow another partition $p_2$ in a ustring for $\cp(n,k)$ if and only if the vertex at the tail of $p_1$ is also at the head of $p_2$.  Thus, there is a bijection between the Eulerian cycles of $G_{n,k}$ and the ustrings of $\cp(n,k)$.

\begin{thm}\label{t:eulerian} 
%\marginal{reworded Theorem to fit with rewording from above}
Let $n,k \in{\bf Z}^+$ with $2 \leq k < n$, and let $G_{n,k}$ be the transition digraph for $\cp(n,k)$.  Then $G_{n,k}$ has an Eulerian cycle.
\end{thm}

\begin{proof}
%\marginal{mass rewording to fit with rewording above, also lots of clarification, **took out first paragraph rewrote the first and second sentences below}%We must show that $G_{n,k}$ is balanced and weakly connected.  We first show that indegree($v$) = outdegree($v$) for all vertices $v$ of $G_{n,k}$.  Suppose that $v$ is a given vertex of $G_{n,k}$, and $(b_1, b_2, \dotsc, b_{n-1})$ is a representation for $v$.  Then for any edge $uv$ we know that $u$ has a representation $(a_0, b_1, b_2, \dotsc, b_{n-2})$ for some $a_0$, and for any edge $vw$ we know that $w$ has a representation $(b_2, b_3, \dotsc, b_{n-1}, c_n)$ for some $c_n$.  If $v$ is a $k-1$ partition of $[n-1]$, then, by the construction of $G_{n,k}$, $a_0$ must be the symbol missing from $\{b_1, b_2, \dotsc, b_{n-1}\}$.  Similarly, $c_n = a_0$, and so indegree($v$) = outdegree($v$) = 1.  If $v$ is a $k$-partition of $[n-1]$, then both $a_0$ and $c_n$ may be any of the $k$ symbols, so in this case we get indegree($v$) = outdegree($v$) = $k$.
$G_{n,k}$ is balanced as remarked above, so we must show that $G_{n,k}$ is weakly connected.
To do so, we show that there exists a path from any vertex of $G_{n,k}$ to the vertex $w$ with representation $(1, 2, \dotsc, k-1, k, k, \dotsc, k)$.  Accordingly, let $u$ be a vertex of $G_{n,k}$.  We describe an algorithm for obtaining a path from $u$ to $w$.  We first find a path from $u$ to a vertex $v$ which ends in $k$ distinct symbols.  We may arrive at such a vertex in $k-1$ steps by a path $u= v_1, v_2, \dotsc v_k$ where, for $i = 1, 2, \dotsc, k-1$, we choose $v_{i+1}$ to be a vertex connected to $v_i$ such that the representations of $v_{i+1}$ end in $i+1$ distinct symbols.  Note that choosing $v_{i+1}$ this way is always possible - $v_i$ will have representations ending in $i$ distinct symbols and if outdegree($v_i$) = 1 then the only possible choice for $v_{i+1}$ has representations formed by adding the missing symbol of each representation of $v_i$ to its last $n-2$ symbols 
%\marginal{**reworded this sentence} 
(the case where outdegree($v_i$)=$k$ is clear).  Now, $v_k$ has representations ending in $k$ distinct symbols, so for any path of length $(n-1) - k$ starting at $v_k$, each vertex on the path will have outdegree = $k$.  Thus, there exists a path $v_k, v_{k+1}, \dotsc, v_{n-1}$, where $v_{k+j}$ has representations whose last $j+1$ symbols are all the same ($j = 0, 1, \dotsc, (n-1) - k$).  Then, by construction, $v_{n-1} = w$. 
%\marginal{**corrected $v_{k+j}$ to $v_{n-2}$}  
Hence, $G_{n,k}$ is weakly connected and so it follows that $G_{n,k}$ contains an Eulerian cycle.
\end{proof}

Hence, we know that Eulerian cycles exist in $G_{n,k}$, and therefore ustrings of $\cp(n,k)$ exist as well.  However, there may be 
%\marginal{fixed typo of ustrings} 
ustrings which cannot be turned into ucycles, which occurs when the representations of the first and last partitions do not overlap correctly, i.e., they have their symbols permuted.  This idea is illustrated in the Eulerian cycle in Figure ~\ref{fig:G_4,3cycle}:

\begin{figure}[h]
	\centering
	\includegraphics[width = .45\textwidth, trim = 0cm 0cm 0cm 0cm, clip = true]{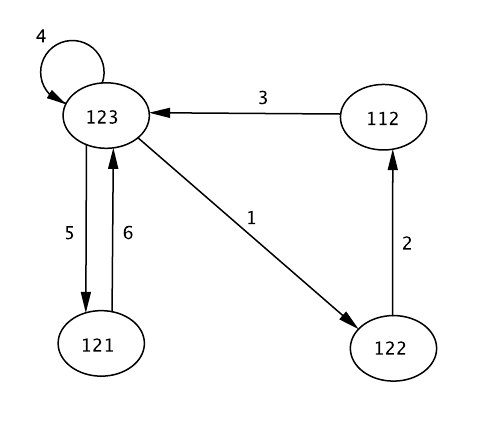}
	\caption{Eulerian cycle in $G_{4,3}$}
	\label{fig:G_4,3cycle}
\end{figure}

%marginal{**fixed typo of ustring, added to sentence} 
If we start with the representation 123 of the first vertex, then this Eulerian cycle represents the ustring 123312132, which cannot be turned into a ucycle.  This example shows another important concept - once we choose the first representation to use, all other representations used are uniquely determined by the given Eulerian cycle.  These observations motivate the following definitions.  %This is because of the previously mentioned fact that each edge, in fact, represents a bijection between the representations of the vertices it connects
%\marginal{commented out sentence, **added a different sentence}

%\marginal{**commented out paragraph, added following definition}
%Recall that when all $k$ symbols are used in a representation, the relative order of the symbols of the representation is a permutation of $[k]$.  If only $k-1$ symbols are used in the representation, then we may add the missing symbol to the end of the relative order to make it a permuation as well.  For ease of notation, in the case that only $k-1$ symbols are used in a representation we use \emph{relative order} to refer to this modified string (with the missing symbol appended) from this point forward.  These observations motivate the following definition.

\begin{defn}
Suppose $v$ is a vertex in $G_{n,k}$ and $r$ is a representation for $v$.  Form a new string $r_0$ from $r$ by deleting all but the first occurence of each symbol from $r$ and appending the missing symbol to the end if $v$ is a $(k-1)$-partition.  Then $r_0$ is a permutation of $[k]$; call $r_0$ the \emph{relative order} of $r$.
\end{defn}

\begin{defn}
Consider an edge $vw$ for some vertices $v$ and $w$ of $G_{n,k}$.  Fix a representation $r_v$ of $v$ and suppose it has relative order $\pi_v$. 
%\marginal{**took out paretheses}  
Suppose the corresponding representation of $w$ is $r_w$ with relative order $\pi_w$. 
%\marginal{**rewrote last sentence} 
Then $\pi_w \pi_v^{-1}$ is called the \emph{associated permutation} of the edge $vw$.
\end{defn}

\begin{rem} 
%\marginal{**added (possibly unnecessary?) remark}
We have defined the associated permutation as the $\pi \in S_k$ such that $\pi \pi_v =\pi_w$, so that this definition is independent of the choice of representation of $v$.
\end{rem}

%\marginal{commented out now out-of-place sentence}
%The fact that this definition is well-defined follows from the discussion in the third and fourth paragraphs of this section.
The graph $G_{5,3}$ is shown again in Figure ~\ref{fig:G_5,3WithPerms} with edges labeled with their associated permutations (expressed in cycle notation with fixed points supressed).

\begin{figure}[h]
	\centering
	\includegraphics[width=1.0\textwidth, trim = 2cm 0cm 2cm 0cm, clip = true]{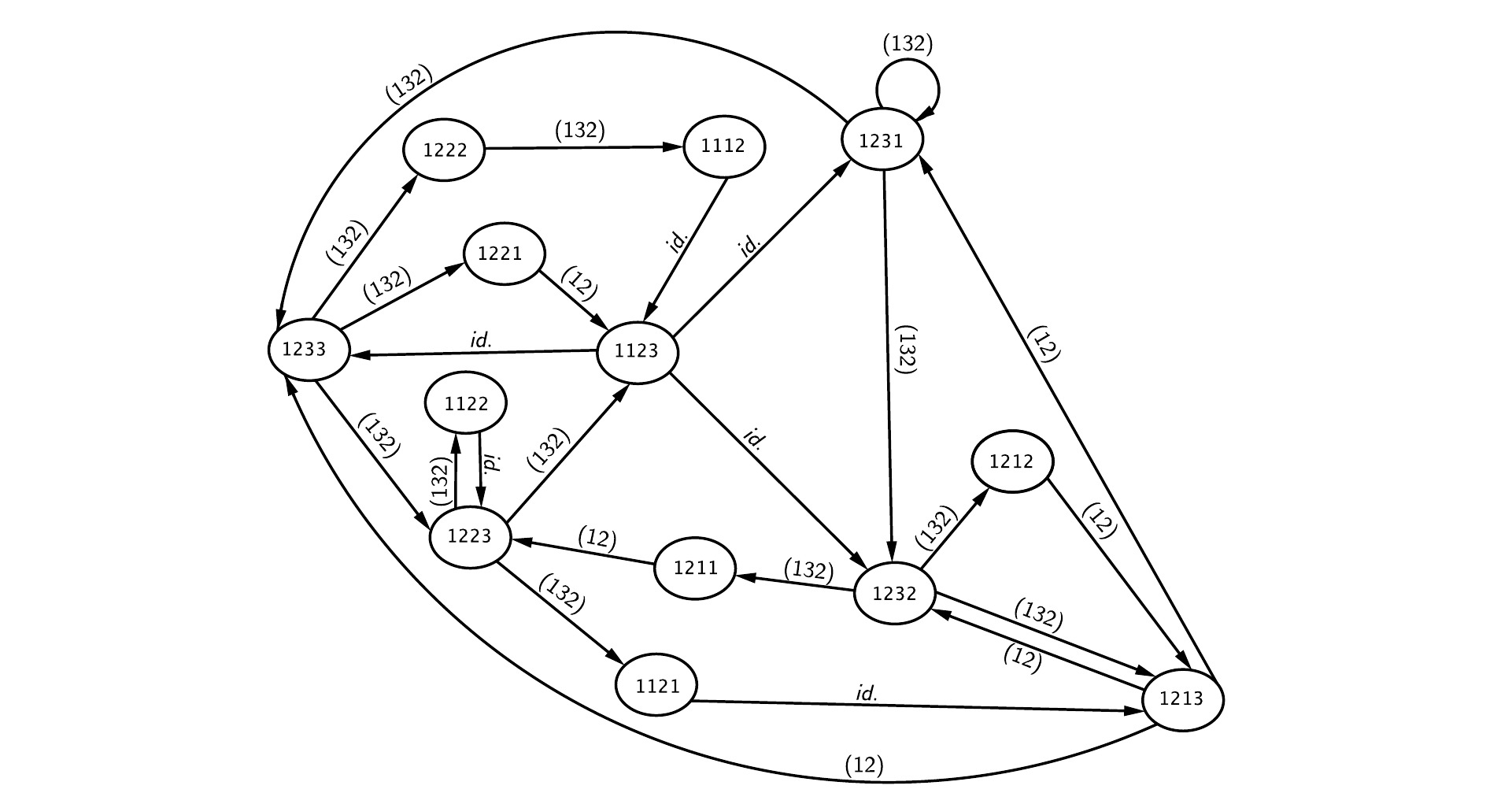}
	\caption{$G_{5,3}$ with associated permutations}
	\label{fig:G_5,3WithPerms}
\end{figure}
%\marginal{commented out somewhat obvious remark}
%Another (possibly obvious, though still noteworthy) consequence of the above observations is the fact that the representations of a given vertex used in an Eulerian cycle may differ at different parts of the cycle.  For example, in our previously considered U-string 123312132 the first vertex uses representations 123, 312, 213, and then 132, in that order.

\begin{defn}
Let $E = e_1, e_2, \ldots, e_{S(n,k)}$ be an Eulerian cycle in $G_{n,k}$ and let $\pi_i$ be the associated permutation of $e_i$, $i = 1, 2, \ldots, S(n,k)$.  We call the product $\pi_{S(n,k)} \pi_{S(n,k) - 1} \cdots \pi_2 \pi_1$ the \emph{permutation product} of $E$.
\end{defn}

From this definition, we get the following characterization.

\begin{thm}\label{t:characterization}
An Eulerian cycle $E = e_1, e_2, \ldots, e_{S(n,k)}$ in $G_{n,k}$ can be lifted to a ucycle of $\cp(n,k)$ if and only if its permutation product is the identity.
\end{thm}
\begin{proof}
%\marginal{fixed representation typo} 
Fix a representation $r$ of the vertex at the tail of $e_1$ and suppose $r$ has relative order $\tau$.  $E$ can be lifted to a ucycle if and only if we arrive back at $r$ at the end of the cycle, and going through $E$ is equivalent to applying the permutation product to $\tau$.
\end{proof}
%\marginal{commented out corollary (and corresponding remark) that was computationally useful but not really relevant to the paper}
%\begin{cor}\label{c:characterizationcorollary}
%An Eulerian cycle $E = e_1, e_2, \ldots, e_{S(n,k)}$ in $G_{n,k}$ can be lifted to a ucycle if and only if $E_i = e_i, e_{i+1}, \ldots, e_{S(n,k)}, e_1, \ldots, e_{i-1}$ can be lifted to a ucycle for all $i \in [S(n,k)]$.
%\end{cor}
%\begin{proof}
%This follows directly from Theorem ~\ref{t:characterization} and the fact that a product is the identity if and only if all cyclic permutations of the product are the identity.
%\end{proof}

%\begin{rem}
%Corollary ~\ref{c:characterizationcorollary} says that when considering whether a given Eulerian cycle can be lifted to a ucycle we may start at any edge in the cycle if we cyclically permute the rest of the edges accordingly.
%\end{rem}

Now, we show that the associated permutation of an edge is completely determined by the vertex at its ``tail", and that only certain permutations can be associated permutations.

\begin{lm}\label{l:associatedperm}
Let $vw_1$ be an edge in $G_{n,k}$, suppose $vw_1$ has associated permutation $\pi$.  Then $\pi$ has the form $(1 \; j \; j\!-\!1 \cdots 2)$ for some $1 \leq j \leq k$, and if $vw_2$ is another edge from $v$, then $vw_2$ has associated permutation $\pi$ as well.
\end{lm}
%(proof might need a little more clarification for cases of $k$ symbols or $k-1$ symbols)
\begin{proof}
Let $r = v_1 v_2 \cdots v_{n-1}$ be the representation of $v$ with relative order $1 2 \cdots k$.  Then there are representations $r_1 = v_2 \cdots v_{n-1} u_1$ and $r_2 = v_2 \cdots v_{n-1} u_2$ of $w_1$ and $w_2$ corresponding to $r$ under $vw_1$ and $vw_2$, respectively.  Suppose that $j-1$ distinct symbols appear in $r$ after $v_1\ (=1)$ and before a second appearance of 1 (1 may only occur once in $r$).  Since the first $n-2$ characters of both $r_1$ and $r_2$ are the same as the last $n-2$ characters of $r$, it follows that $r_1$ and $r_2$ both have relative order $2 3 \cdots j \; 1 \; j\!+\!1 \; j\!+\!2 \cdots k$.  Hence, the associated permuations of $vw_1$ and $vw_2$ are both $(1 \; j \; j\!-\!1 \cdots 2)$.
\end{proof}

%%%%%%%%%%%To see this, suppose a given vertex has representations where the first symbol appears for a second time after the appearance of $j-1$ other symbols (where if this symbol only appears once it is said to appear for a second time after all $n-1$ other symbols (need to explain when a symbol is missing)).  Then when we follow any edge from this vertex, we lose the first character in the string and append some other character to the end of the string.  The symbol which previously appeared first now appears in the $j$ position relative to the first appearance of the other symbols.  Also, if any other symbol appeared in the $l$ position previously ($l \neq 1$), then it now appears in the $l-1$ position if $l \leq j$ or in the $l$ position if $l > j$.  This give the edge the associated permutation $(1 \; j \; j\!-\!1 \cdots 2)$.

\begin{thm}\label{t:k=2}
For $n \geq 3$, every Eulerian cycle of $G_{n,2}$ can be lifted to a ucycle.
\end{thm}

\begin{proof}
Observe that the vertex with representation $11 \cdots 1$ is the only vertex of outdegree 1 and that the edge coming out of this vertex has $id.$ as its associated permutation.  All other vertices have outdegree 2, and the two edges originating from any particular vertex both have the same associated permutation by Lemma ~\ref{l:associatedperm}.
%\marginal{added the word lemma}.  
In particular, there is an even number of $\pmt{(12)}$ permutations.  Since $S_2$ is abelian, the permutation product of an Eulerian cycle will be the identity so that the result follows by Theorem ~\ref{t:characterization}.
\end{proof}

\begin{cor}
For $n \geq 3$, ucycles of $\cp(n,2)$ exist.
\end{cor}
\begin{proof}
This follows directly from Theorem ~\ref{t:eulerian} and Theorem ~\ref{t:k=2}.
\end{proof}

We can also use the permutation product to determine cases when ucycles do not exist.  The easiest way for this to occur is if the multiset consisting of all associated permutations in $G_{n,k}$ contains an odd number of odd permutations since this ensures that there is no ordering of the associated permutations which multiplies to the identity.

\begin{defn}
We call the multiset consisting of all associated permutations in $G_{n,k}$ the \emph{permutation multiset} of $G_{n,k}$.
\end{defn}

\begin{defn}\label{d:parityfunction}
Let $\mathcal{O}$ be the multiset which contains all odd permutations of the permutation multiset of $G_{n,k}$.  Define the \emph{parity function} by
\begin{displaymath}
	Par(n,k) = \left\{
		\begin{array}{lr}
			0 & \text{\emph{if }} |\mathcal{O}| \equiv 0 \text{\emph{ mod 2}}\\
			1 & \text{\emph{if }} |\mathcal{O}| \equiv 1 \text{\emph{ mod 2}}
		\end{array}
	\right.
\end{displaymath}
%(maybe write $Par : \Z_+ \times \Z_+ \rightarrow \mathbb{F}_2$, $Par(n,k) \equiv |\mathcal{O}| \text{\emph{ mod 2}}$ instead.)
\end{defn}
%\marginal{commented out trivial remark}
%\begin{rem}
%Note that $Par(n,k) = 0$ if the number of odd permutations in the permutation multiset is even, and $Par(n,k) = 1$ if the number of odd permutations in the permutation multiset is odd.
%\end{rem}

\begin{lm}\label{l:oddparity}
If $Par(n,k) = 1$, then there does not exist a ucycle of $P(n,k)$.
\end{lm}

The following formula gives a recursive formula for calculating $Par(n,k)$.
%(May need better initial conditions/description)

%%%%%%%%%%%%%NEED TO UPDATE LEMMA AND PROOF

\begin{lm}\label{l:parityrecursion}
The function $Par(n,k)$ satifies the following recurrence relation: 
\begin{equation}\label{e:parityrecursion}
Par(n,k) \equiv k \cdot Par(n-1,k) + Par(n-1,k-1) + S(n-2,k-2) \text{\emph{ mod 2}}
\end{equation}
with initial conditions $Par(n,2) = 0$ for all $n$, and
$$Par(n,n-1) = 
\begin{cases}
1 & \text{\emph{ if }} n \equiv 0 \text{\emph{ mod 4}}\\
0 & \text{\emph{ otherwise}}
\end{cases}$$
\end{lm}

\begin{proof}
We establish a relationship between the 
%\marginal{deleted ``associated"} 
permutation multiset of $G_{n,k}$ and those of $G_{n-1,k}$ and $G_{n-1,k-1}$.  Suppose $v$ is a vertex in $G_{n,k}$, so $v$ represents a $k$ or $k-1$-partition $p$ of $[n-1]$.  We consider the edge $e_p$ which represents $p$ in either $G_{n-1,k}$ or $G_{n-1,k-1}$.  We know that the associated permutation of $e_p$ is determined by the location of the second occurence of the first symbol in a representation $r_p$ of the vertex $w_p$ at the tail of $e_p$ by Lemma ~\ref{l:associatedperm}. 
% \marginal{this proof - and this paragraph in particular - might need some clarification}
First, suppose $v$ represents a $k$-partition of $[n-1]$.  If the first symbol does actually occur for a second time in $r_p$, then since there is a representation of $v$ whose first $n-2$ characters are precisely $r_p$, it follows that $e_p$ has the same associated permutation as all the edges coming from $v$.  If the first symbol of $r_p$ does not occur a second time, then $e_p$ has associated permutation $(1 \; k \; k\!-\!1 \cdots 2)$.  If $w_p$ has outdegree 1, then the representations of $v$ do not have a second occurrence of their first symbols, and so all edges from $v$ have associated permutation $(1 \; k \; k\!-\!1 \cdots 2)$.  If $w_p$ has outdegree $k$, then the representations of $v$ have an occurence of all symbols before a second occurence of the first symbol, so we get that the edges form $v$ have associated permutation $(1 \; k \; k\!-\!1 \cdots 2)$ again.  Since each vertex of $G_{n,k}$ which represents a $k$-partition of $[n-1]$ has outdegree $k$, we get the term $k \cdot Par(n-1,k)$.

Now, suppose $v$ represents a $k-1$-partition of $[n-1]$.

\emph{Case 1, the first symbol of $v$ appears a second time.}  Then either the first symbol of $r_p$ appears a second time, or the first symbol of $r_p$ is appended by following $e_p$ (if the second appearance in $v$ is at the last character).  If the first symbol of $r_p$ appears a second time, then by previous reasoning $e_p$ has the same associated permutation as all edges from $v$.  If the first symbol of $r_p$ does not occur a second time, then $e_p$ and the edges from $v$ all have associated permutation $(1 \; k \; k\!-\!1 \cdots 2)$.

\emph{Case 2, the first symbol of $v$ does not appear a second time.}  Then the first symbol of $r_p$ does not appear a second time, and so $e_p$ must have associated permutation $(1 \; k\!-\!1 \; k\!-\!2 \; \cdots 2)$.  However, in this case $v$ has associated permutation $(1 \; k \; k\!-\!1 \cdots 2)$.  Note that since the first symbol of $v$ does not appear a second time, the last $n-2$ characters of $v$ represent a $k-2$-partition of $[n-2]$, so this case occurs exactly $S(n-2,k-2)$ times.  Thus, we have $S(n-2,k-2)$ partitions that either switch from even to odd or odd to even; in either case adding $S(n-2,k-2)$ affects the parity in the desired manner.

Thus, each vertex in $G_{n,k}$ which represents a $k-1$-partition of $[n-1]$ has the same associated permutation as it does in the graph $G_{n-1,k-1}$ except for $S(n-2,k-2)$ permutations which change sign.  Since each such vertex has outdegree 1, we get the $Par(n-1,k-1) + S(n-2,k-2)$ term.\\

Finally, the initial condition $Par(n,2) = 0$ for all $n$ follows from Theorem ~\ref{t:k=2}, and the initial condition
$$Par(n,n-1) = 
\begin{cases}
1 & \text{if } n \equiv 0 \text{ mod 4}\\
0 & \text{ otherwise}
\end{cases}$$
follows from the remark following Theorem 5 in ~\cite{cs} 
%(not quite, may need better justification).
\end{proof}

%(might need further explanation of when certain vertices have certain permutations to make above proof more valid)

\begin{cor}\label{c:parityformula}
For $n \geq 4$ and $2 \leq k < n$,
\begin{equation}\label{e:parityformula}
	Par(n,k) \equiv
	\begin{cases}
		0 \text{\emph{ mod 2}} & \text{\emph{if }} n \text{\emph{ is odd}}\\
		S(n-2,k-2) \text{\emph{ mod 2}} & \text{\emph{if }} n \text{\emph{ is even}}
	\end{cases}
\end{equation}
\end{cor}
%\marginal{changed all $f$'s to $Par$'s}
\begin{proof}
We proceed by induction.  For the base cases, we show that the initial conditions of the recursion in Corollary ~\ref{l:parityrecursion} satisfy Equation (~\ref{e:parityformula}).  If we define $S(n,0) := 0$ for all $n$, then Equation (~\ref{e:parityformula}) yields 
%\marginal{changed 0 in parentheses to 2} 
$Par(n,2) = 0$ for all $n$.  Also, if $n$ is even, then Equation ~\ref{e:parityformula} yields
\begin{eqnarray*}
	Par(n,n-1) = S(n-2, n-3) = \binom {n-2}{2} &=& \frac{(n-2)(n-3)}{2}\\ &\equiv&
	\begin{cases}
		0 \text{ mod 2} & \text{if } n \equiv 2 \text{ mod 4}\\
		1 \text{ mod 2} & \text{if } n \equiv 0 \text{ mod 4}
	\end{cases}
\end{eqnarray*}
as desired.

Now, suppose that (\ref{e:parityformula}) holds for $Par(n-1,k)$ and $Par(n-1,k-1)$.  If $n$ is odd then $n-1$ is even, so by (\ref{e:parityrecursion}), the induction hypothesis, and the fundamental Stirling number recurrence, 
\begin{align*}
	Par(n,k) & \equiv k \cdot Par(n-1,k) + Par(n-1,k-1) + S(n-2,k-2) & \text{ mod 2}\\
		 & \equiv k \cdot S(n-3,k-2) + S(n-3,k-3) + S(n-2,k-2) & \text{ mod 2}\\
		 & \equiv 2 S(n-3,k-2) + S(n-2,k-2) + S(n-2,k-2) & \text{ mod 2}\\
		 & \equiv 2 S(n-2,k-2) &\text{ mod 2}\\
		 & \equiv 0 & \text{ mod 2}
\end{align*}
If $n$ is even, then $n-1$ is odd and so $Par(n-1,k) = Par(n-1,k-1) = 0$.  Hence, $Par(n,k) \equiv S(n-2,k-2)$ mod 2 by (\ref{e:parityrecursion}).
\end{proof}

\begin{cor}\label{c:oddparitycondition}
If $n$ is even and $S(n-2,k-2)$ is odd, there does not exist a ucycle for $k$-partitions of $[n]$.
\end{cor}

\begin{proof}
This follows directly from Corollary ~\ref{c:parityformula}.
\end{proof}

\begin{cor}\label{c:346}
If $n \geq 4$ is even, then there does not exist a ucycle for $3, 4,$ or $6$-partitions of $n$.
\end{cor}

\begin{proof}
By Corollary ~\ref{c:oddparitycondition}, we must prove that $S(n,1), S(n,2),$ and $S(n,4)$ are odd for all even $n$.  We know $S(n,1) = 1$ and $S(n,2) = 2^{n-1} - 1$, which are both odd.  Now, 
\begin{align*}
	S(n,4) & = \frac{1}{4!} \left[ - \binom {4}{1} 1^n + \binom {4}{2} 2^n - \binom {4}{3} 3^n + \binom {4}{4} 4^n \right] \\
		  & = \frac{1}{4!} (-4 + 6 \cdot 2^n - 4 \cdot 3^n + 4^n)\\
		  & = \frac{1}{6} (-1 + 3 \cdot 2^{n-1} - 3^n + 4^{n-1})
\end{align*}
If we consider the term in parentheses mod 4, then we get $(-1 + 0 -1 + 0) \equiv 2$ ($3^n \equiv 1$ mod 4 for even $n$).  Thus 2 only divides the term in parentheses once, so when we divide it by 6 we get an odd number.
\end{proof}

\section{Open Questions}   In addition to the question raised in the last line of the proof of Theorem 1, we can ask the following:

(a) The smallest case remaining after our investigation is the one for $n=5,k=3$, for which we have found the ucycle 3112311123213233112131322.  All other cases for $n = 5$ and $n = 6$ are solved.  This leads to the question:  What is the best result that can be proved along the following lines ``For $n \geq 3$ and $3 \leq k < n$, ucycles of $k$-partitions of $[n]$ exist if and only if $n$ is odd?"  

(b) In general, how can one use Theorem 10 to prove results on {\it existence} of ucycles (rather than non-existence)?

(c) Even if ucycles of $\cp(n,k)$ may not exist, when it it true that ucycles exist for $\cp(n,s,t)$, the set of partitions of $[n]$ into between $s$ and $t$ parts; $s<t$?  

(d) Throughout this paper, we have insisted on having the alphabet size equal $k$.  How do our results change if we relax this condition?

\section{Acknowledgments}  The research of all the authors was supported by NSF Grant 1263009.

\end{document}